\documentclass[12pt, reqno]{amsart}

\usepackage{version}
\usepackage[T1]{fontenc} 
\usepackage[utf8]{inputenc} 
\usepackage{amsmath}
\usepackage{amssymb}
\usepackage{mathrsfs}
\usepackage{amsthm}
\usepackage{latexsym}
\usepackage{wasysym}
\usepackage{dcpic, pictex}
\usepackage{graphicx}
\usepackage{tikz-cd}

\newtheorem{theorem}{Theorem}[section]
\newtheorem{lemma}[theorem]{Lemma}
\newtheorem{corollary}[theorem]{Corollary}
\newtheorem{proposition}[theorem]{Proposition}

\theoremstyle{definition}
\newtheorem{definition}[theorem]{Definition}
\newtheorem{example}[theorem]{Example}

\newtheorem{remark}[theorem]{Remark}
\newtheorem{question}{Question}

\author{Matteo Fiacchi}

\def\O{{\mathcal{O}}}

\def\D{{\mathbb{D}}}
\def\C{{\mathbb{C}}}
\def\R{{\mathbb{R}}}

\def\Z{{\mathbb{Z}}}
\def\N{{\mathbb{N}}}

\def\Bn{{\mathbb{B}^{n}}}
\def\Cn{{\mathbb{C}^{n}}}

\newcommand{\id}{{\sf Id}}
\newcommand{\re}{{\sf Re}}
\newcommand{\im}{{\sf Im}}

\begin{document}
\title{Approximation and Loewner Theory of holomorphic covering mappings} 

\address{M. Fiacchi: Dipartimento di Matematica\\
Universit\`{a} di Roma \textquotedblleft Tor Vergata\textquotedblright\ \\Via Della Ricerca Scientifica 1, 00133 \\
Roma, Italy} \email{fiacchi@mat.uniroma2.it}

\begin{abstract}
We give conditions in order to approximate locally uniformly holomorphic covering mappings of the unit ball of $\Cn$ with respect to an arbitrary norm, with entire holomorphic covering mappings. The results rely on a generalization of the Loewner theory for covering mappings which we develop in the paper.
\end{abstract}
\maketitle

\section{Introduction}

It is a well known fact that in one complex variable the automorphisms group of the complex plane is composed only by affine transformations. Conversely, in higher dimensions the automorphisms group of the Euclidean space is so large to allow approximation of univalent mappings: we state here the Andérsen-Lempert Theorem \cite{AL} in a suitable form for this paper (the result holds in fact for general starlike domains).

\begin{theorem}[Andérsen-Lempert]\label{AL}
For $n\geq2$, let $B=\{z\in\Cn:||z||<1\}$ be the unit ball of $\Cn$ with respect to an arbitrary norm and let $f:B\longrightarrow \Cn$ be an univalent mapping. If there exists a neighborhood $\Lambda$ of $f(B)$ biholomorphic to $\Cn$ (that is, $\Lambda$ is either $\Cn$ itself or a Fatou-Bieberbach domain) such that $(f(B),\Lambda)$ is Runge, then $f$ can be approximated uniformly on compacta of $B$ by univalent mappings of $\Cn$, i.e. for every compact subset $K$ of $B$, $\epsilon>0$ there exists an automorphism of $\Cn$ or a Fatou-Bieberbach mapping $\Psi$ such that $\max_{z\in K}||f(z)-\Psi(z)||<\epsilon$.
\end{theorem}

This result ensures the approximability of univalent mappings by means of entire univalent mappings (i.e. defined on all $\Cn$) under the assumption of the existence of specific neighborhoods. Another approach to this kind of problems is using Loewner theory, a dynamical method in geometric function theory introduced in dimension one by Charles Loewner in the 1920's \cite{LW}. The main idea of Loewner was to construct a one parameter family of univalent functions whose images are increasing domains, and study them using differential equations \cite{GK}. Using this theory we have the following formulation of Theorem \ref{AL} (in a certain sense it is an equivalent statement, see \cite{FM1} for a discussion about this).

\begin{theorem}
For $n\geq2$ let $f:B\longrightarrow\Omega\subseteq\Cn$ an univalent mapping. If $f$ embeds into Loewner chain, i.e. there exists a Loewner chain $(f_t)_{t\geq0}$ such that $f=f_0$, then there exists a sequence of entire univalent mappings which converges to $f$, uniformly on compacta of $B$.
\end{theorem}

It is natural and a completely open question whether similar results hold for locally univalent functions. 

In this paper we focus on a particular class of locally univalent functions: holomorphic covering mappings. The main result of the paper is the following.

\begin{theorem}\label{apprxcov}
For $n\geq2$ let $f:B\twoheadrightarrow\Omega\subseteq\Cn$ be a holomorphic covering mapping. If $f$ embeds into a Loewner chain of covering mappings, i.e. there exists a Loewner chain of covering mappings $(f_t)_{t\geq0}$ such that $f=f_0$, then there exists a sequence of entire holomorphic covering mappings  which converges to $f$, uniformly on compacta of $B$.
\end{theorem}

Such a result is based on a general Loewner theory for holomorphic covering mappings which we develop in the paper, and on a abstract approximation theorem (see Theorem \ref{genapcov}). We also provide explicit examples of applications of our construction (see Example \ref{exball} and Theorem \ref{aprxpol}).

\section{Classical Loewner theory}

We give a fast overview of Loewner theory (for more details see for instance \cite{GK} and \cite{PMK}). 

Let $B:=\{z\in\Cn: ||z||<1 \}$ be the unit ball of $\Cn$ with respect to an arbitrary norm $||\cdot||$. A \textit{(normalized) Loewner chain} is a family of univalent mappings $(f_t:B\longrightarrow\Cn)_{t\geq0}$  with $\Omega_s:=f_s(B)\subseteq f_t(B)$ for $s\leq t$ and such that $f_t(0) = 0$ and $(df)_0=e^t\id$ for all $t\geq 0$.
We also define the \textit{Loewner range} of a Loewner chain as
\[
R(f_t)=\bigcup_{t\geq0}\Omega_t\subseteq\Cn.
\]
The Loewner range is always biholomorphic to $\Cn$,  although, for $n>1$, it can be strictly contained in $\Cn$ (see \cite{ABW}). Furthermore a Loewner chain $(f_t)_{t\geq0}$ is \textit{normal} if $\{e^{-t}f_t(\cdot)\}_{t\geq0}$ is a normal family.

We recall the following result

\begin{proposition}\cite{GK}\label{p1}
If $(f_t)_{t\geq0}$ is a Loewner chain, then there exist an unique normalized biholomorphism $\Psi:\Cn \longrightarrow R(f_t)$ and normal Loewner chain $(g_t)_{t\geq0}$ such that for each $t\geq0$
$$f_t=\Psi \circ g_t. $$
\end{proposition}

Let $(f_t)_{t\geq0}$ be a Loewner chain. We define for each $0\leq s\leq t$
$$\phi_{s,t}=f_t^{-1}\circ f_s. $$

The family $(\phi_{s,t})_{0\leq s\leq t}$ is called the \textit{evolution family associated to the chain $(f_t)_{t\geq0}$}: it has the following properties 

(EF1) $ (d\phi_{s,t})_0=e^{s-t}\id$;

(EF2) $\phi_{s,s}=\id_{B}$ for each $s\geq0$;

(EF3) $\phi_{s,t}=\phi_{u,t}\circ  \phi_{s,u}$ for each $0\leq s\leq u\leq t$.\\
Evolution families are well studied, with a strong link with Loewner theory and semicomplete vector fields (see for instance \cite{GK}, \cite{ABHK} and \cite{ABinfgen}). From the properties (EF1-3) we have that $(\phi_{s,t})_{0\leq s\leq t}$ are locally uniformly Lipschitz \cite{ABinfgen}, i.e. for each $T>0$ and $r\in(0,1)$ there exists a positive constant $M(T,r)$ such that 
$$ || \phi_{s,t}(z)-\phi_{s,u}(z)||\leq M(T,r)|t-u|, \ z\in rB, \ 0\leq s\leq u\leq t \leq T,$$
and this implies the continuity $(f_t)_{t\geq0}$ with respect to $t$.

Furthermore, for each $0\leq s\leq t$ the mappings $\phi_{s,t}$ are univalent \cite{ABHK}.
Finally we have the following 

\begin{proposition}\cite{GHK}\label{p2}
Let $(\phi_{s,t})_{0\leq s\leq t}$ be a family of self mappings of $B$ such that (EF1-3) hold, then there exists an unique normal Loewner chain $(g_t)_{t\geq0}$ such that $\phi_{s,t}=g_t^{-1}\circ g_s$ for each $0\leq s\leq t$.
\end{proposition}

\section{Covering mappings}
We recall some basic notions.

Let $D$ and $\Omega$ be two complex manifolds, a local biholomorphism $f:D\twoheadrightarrow\Omega$ is an \textit{holomorphic covering mappings} (or, for short in this paper, just a \textit{covering mappings}) if for each path $\gamma$ in $\Omega$ with initial point $x$ and for each point $y\in f^{-1}(x)$, there exists an unique path $\tilde{\gamma}$ in $D$ with initial point $y$ such that $f\circ\tilde{\gamma}=\gamma$.

Let $f:D\twoheadrightarrow\Omega\subseteq\Cn$ be a covering mapping, we denote with $DT(f):=\{\Psi\in \text{Aut}(D) : f\circ\Psi=f \}$ the \textit{group of deck transformations of $f$}. We recall that $DT(f)$ is a subgroup of Aut($D$) that acts freely and properly discontinuous on $D$.

In this section we give some useful results concerning covering mappings.

\begin{remark}\label{covuni}
Let $\Omega\subset\Cn$, with $0\in\Omega$, be a domain covered by $B$, then there exists at most one covering mapping $f:B\twoheadrightarrow \Omega$ such that $f(0)=0$ and $df_0=\lambda\id$ for some $\lambda>0$. Indeed, let $\tilde{f}:B\twoheadrightarrow \Omega$ be an another covering mapping such that $\tilde{f}(0)=0$ and $d\tilde{f}_0=\tilde{\lambda}\id$, then for the uniqueness of the universal covering there exists an automorphism $\Psi$ of $B$ such that $\Psi(0)=0$, $d\Psi_0=\lambda\tilde{\lambda}^{-1}\id$ and $f=\tilde{f}\circ\Psi$.

\[
\begin{tikzcd}[column sep=2pc,row sep=3pc]
\textrm{$B$} \arrow[rr,rightarrow,"\Psi"] \arrow[rd,rightarrow,"f"] & & \textrm{$B$} \arrow[ld,rightarrow,"\tilde{f}"] \\
 & \textrm{$\Omega$} &
\end{tikzcd}
\]
Now, by Cartan uniqueness theorem for bounded domains, $\Psi$ has to be the identity and then $f=\tilde{f}$.  
\end{remark}

We also need the following topological lemma.

\begin{lemma}\label{homeolift}
Let $Y$ be a connected topological manifold and $X\subseteq Y$ be connected open set. Let $f:\tilde{X}\twoheadrightarrow X$ and $\Psi:\tilde{Y}\twoheadrightarrow Y$ be two universal topological covering mappings and consider $g:\tilde{X}\longrightarrow \tilde{Y}$ a lifting of $f$ with respect to $\Psi$. Let $i:X\rightarrow Y$ denote the continuous inclusion.
\[
\begin{tikzcd}[column sep=3pc,row sep=3pc]
\textrm{$\tilde{X}$} \arrow[r,rightarrow,"g"] \arrow[d,rightarrow,"f"] & \textrm{$\tilde{Y}$} \arrow[d,rightarrow,"\Psi"] \\
\textrm{$X$} \arrow[r,rightarrow,"i"] & \textrm{$Y$}
\end{tikzcd}
\]
If the morphism $i^*:\pi_1(X)\longrightarrow\pi_1(Y)$ is injective then $g$ is injective.
\end{lemma}
\proof 
Let $x_0,x_1$ be two points in $\tilde{X}$ such that $g(x_0)=g(x_1)$, we have to prove that $x_0=x_1$. Consider a path $\gamma$ such that $\gamma(0)=x_0$ and $\gamma(1)=x_1$: notice that the path $f\circ\gamma$ is closed in $X$. Consider its unique lifting $\tilde{\gamma}$ with respect to $\Psi$ such that $\tilde{\gamma}(0)=g(x_0)$, by the uniqueness of the lifting we have $\tilde{\gamma}=g\circ\gamma$. Now
$$\tilde{\gamma}(0)=g(\gamma(0))=g(x_0)=g(x_1)=g(\gamma(1))=\tilde{\gamma}(1)$$
i.e. $\tilde{\gamma}$ is a closed path, then $f\circ\gamma$ is homotopic to the constant path in $Y$. By the injectivity of $i^*$, $f\circ\gamma$ is also homotopic to the constant path in $X$, and then its lifting $\gamma$ is a closed path, i.e. $x_0=x_1$.
\endproof

\section{Loewner chains of covering mappings}

Now we can give the notion of Loewner chain of covering mappings.

\begin{definition}
A \textit{Loewner chain of covering mappings} is a family $(f_t)_{t\geq0}$ of holomorphic mappings on $B$ such that

1) $f_t:B\rightarrow\Cn$ is a covering mapping on a pseudoconvex domain $\Omega_t$ for each $t\geq0$;

2) $f_t(0)=0$, $(df_t)_0=e^t\id$ for each $t\geq0$;

3) $\Omega_s\subseteq\Omega_t$ for each $s\leq t$.\\
Finally, we define the \textit{Loewner range} as the set $R(f_t):=\bigcup_{t\geq0}\Omega_t\subseteq\Cn$.
\end{definition}

\begin{example}\label{annuex}
Let $g$ be the holomorphic function on the unit disk given by
$$g(z)=\frac{i}{2}\log\Bigl(\frac{1-iz}{1+iz}\Bigr), $$
it is a biholomorphism between the disk $\D$ and the vertical strip $\{z\in\C: |\re[z]|<\frac{\pi}{4}\}$. Consider the normalized covering mapping $\Psi(z)=e^z -1$ from $\C$ to $\C\backslash\{-1\}$. Now, the family
$$f_t(z):=\Psi(e^tg(z))=\Bigl(\frac{1-iz}{1+iz}\Bigr)^{\frac{i}{2}e^t}-1$$
is a Loewner chain of covering mappings with images the annuli $A_t:=\{z\in\C: r_t^{-1}<|z+1|<r_t \}$ where $r_t:=e^{\frac{\pi}{4}e^t}$, and Loewner range $\C\backslash\{-1\}$.
\end{example}

For each $s\leq t$ we can lift  the inclusion $i_{s,t}$ of $\Omega_s$ into $\Omega_t$ with respect to $f_t$ and so we can construct an holomorphic function $\phi_{s,t}:B\longrightarrow B$ such that $\phi_{s,t}(0)=0$ and $f_s=f_t\circ \phi_{s,t}$.

\[
\begin{tikzcd}[column sep=3pc,row sep=3pc]
\textrm{$B$} \arrow[r,rightarrow,"\phi_{s,t}"] \arrow[d,rightarrow,"f_s"] & \textrm{$B$} \arrow[d,rightarrow,"f_t"] \\
\textrm{$\Omega_s$} \arrow[r,rightarrow,"i_{s,t}"] & \textrm{$\Omega_t$}
\end{tikzcd}
\]

The family $(\phi_{s,t})_{0\leq s\leq t}$ has the following properties.

\begin{proposition}\label{efcov}
Let $(f_t)_{t\geq0}$ be a Loewner chain of covering mappings, and let $(\phi_{s,t})_{0\leq s\leq t}$ be a family of holomorphic self mappings of $B$ such that $\phi_{s,t}(0)=0$ and $f_s=f_t\circ \phi_{s,t}$ for each $0\leq s\leq t$, then

1) $ (d\phi_{s,t})_0=e^{s-t}\id$;

2) $\phi_{s,s}=\id_{B}$ for each $s\geq0$;

3) $\phi_{s,t}=\phi_{u,t}\circ  \phi_{s,u}$ for each $0\leq s\leq u\leq t$.\\
Therefore $(\phi_{s,t})_{0\leq s\leq t}$ is an evolution family.
\end{proposition}

\proof
1) The functions $f_t$, $f_s$ and $\phi_{s,t}$ fix the origin, then using the chain rule in the expression $f_s=f_t\circ \phi_{s,t}$ we obtain the statement.

2) It follows trivially from the uniqueness of the lifting.

3) By construction we have that for $0\leq s\leq u\leq t$ 
$$f_s=f_u\circ \phi_{s,u} \text{ \ \  and \ \  } f_u=f_t\circ \phi_{u,t}$$
and then
$$f_s=f_t\circ \phi_{u,t}\circ\phi_{s,u}.$$
Furthermore $f_s=f_t\circ \phi_{s,t}$, indeed for the uniqueness of the lifting we obtain the thesis.

\[
\begin{tikzcd}[column sep=3pc,row sep=3pc]
\textrm{$B$}  \arrow[rr, bend left=35, "\phi_{s,t}"] \arrow[r,rightarrow,"\phi_{s,u}"] \arrow[d,rightarrow,"f_s"] & \textrm{$B$} \arrow[r,rightarrow,"\phi_{u,t}"] \arrow[d,rightarrow,"f_u"] & \textrm{$B$} \arrow[d,rightarrow,"f_t"] \\
\textrm{$\Omega_s$} \arrow[r,rightarrow,"i_{s,u}"] & \textrm{$\Omega_u$} \arrow[r,rightarrow,"i_{u,t}"] & \textrm{$\Omega_t$}
\end{tikzcd}
\]
\endproof

Recalling the results show in the Section 3, the functions $\phi_{s,t}$ are univalent and locally uniformly Lipschitz.
\\

\begin{lemma}\label{teclem}
Let $(f_t)_{t\geq0}$ be a Loewner chain of covering mappings and $(\phi_{s,t})_{0\leq s \leq t}$ the associated evolution family. Let $\gamma:[0,1]\longrightarrow\Omega_t$ be a path with $\gamma(0)=0$ and $\tilde{\gamma}$ its lifting with respect to $f_t$ such that $\tilde{\gamma}(0)=0$. If $\gamma([0,1])\subseteq\Omega_s$ for some $s\in[0,t]$, then $\tilde{\gamma}([0,1])\subseteq\phi_{s,t}(B)$.
\end{lemma}
\proof
By assumptions $\gamma([0,1])\subseteq\Omega_s$, then let $\sigma$ be the unique lifting of $\gamma$ with respect to $f_s$ such that $\sigma(0)=0$. Now consider the two paths $\tilde{\gamma}$ and $\phi_{s,t}\circ\sigma$: we have
$$f_t\circ(\phi_{s,t}\circ\sigma)=f_s\circ\sigma=\gamma=f_t\circ\tilde{\gamma}$$
and they have the same initial point, then by the uniqueness of the lifting $\tilde{\gamma}=\phi_{s,t}\circ\sigma$, that implies $\tilde{\gamma}([0,1])\subseteq\phi_{s,t}(B)$.
\endproof

At this point, it is natural to study the evolution of fundamental group of the images of the chain.

\begin{proposition}\label{p1iso}
Let $(f_t)_{t\geq0}$ be a Loewner chain of covering mappings, then for each $0\leq s \leq t $ the morphisms
$$i^*_{s,t}:\pi_1(\Omega_s)\longrightarrow \pi_1(\Omega_t) $$
and
$$i^*_{t,\infty}:\pi_1(\Omega_t)\longrightarrow \pi_1(R(f_t))$$
are injective morphisms of groups.
\end{proposition}
\proof
We start with the case $0\leq s \leq t<\infty$.

We prove that the kernel of $i^*_{s,t}$ is trivial: let $\gamma:[0,1]\rightarrow \Omega_s$ be a closed path that is homotopic to the constant one in $\Omega_t$, we have to prove this is also true in $\Omega_s$. Now $\gamma$ can be lifted to a path $\tilde{\gamma}$ in $B$ such that $f_t\circ \tilde{\gamma}=\gamma$ and $\tilde{\gamma}(0)=\tilde{\gamma}(1)=0$. By Lemma \ref{teclem} $\tilde{\gamma}([0,1])\subseteq\phi_{s,t}(B)$ and so there exists an homotopy $\tilde{G}:[0,1]\times[0,1]\rightarrow \phi_{s,t}(B)$ such that $\tilde{G}(\cdot,0)=\tilde{\gamma}$ and $\tilde{G}(\cdot,1)\equiv0$. Now $G:=f_t\circ\tilde{G}$ is an homotopy in $\Omega_s$ between $\gamma$ and the constant path.

Finally, in the second case, every closed path in $R(f_t)$ is contained in $\Omega_T$ for $T$ big enough, then we can use the homotopy constructed before. 
\endproof

Now we can give an analogous of Proposition \ref{p1} for Loewner chains of covering mappings.

\begin{proposition}\label{fact}
If $(f_t)_{t\geq0}$ is a Loewner chain of covering mappings, then $R(f_t)$ is covered by $\Cn$ and there exist an unique normalized covering mapping $\Psi:\Cn \longrightarrow R(f_t)\subseteq\Cn$ and normal Loewner chain $(g_t)_{t\geq0}$ such that for each $t\geq0$
\begin{equation}\label{facteq}
f_t=\Psi \circ g_t.
\end{equation}
\end{proposition}
\proof
Let $(\phi_{s,t})_{0\leq s \leq t}$ be the evolution family associated to $(f_t)_{t\geq0}$. By Proposition \ref{p2}, there exists a normal Loewner chain $(g_t)_{t\geq0}$ such that $\phi_{s,t}:=g^{-1}_t\circ g_s$ and $R(g_t)=\Cn$, then for each $z\in\Cn$ exists $t\geq0$ such that $z\in g_t(B)$, therefore we can define 
$$\Psi(z):=(f_t\circ g^{-1}_t)(z)$$
and thanks to the relation $f_s=f_t\circ \phi_{s,t}$ the definition does not depend on the choice of $t$.
Then for each $s\geq0$ and $z\in B$ we obtain 
\[f_s(z)=\Psi(g_s(z)).\] 
It remains to prove that $\Psi$ is a covering mapping. For each path $\gamma$ in $R(f_t)$ and $w\in\Psi^{-1}(\gamma(0))$ there exists $t\geq0$ such that $\gamma([0,1])\subseteq f_t(B)$ and $w\in g_t(B)$: the lifting of $\gamma$ with respect to $\Psi$ with initial point $w$ is $g_t\circ\tilde{\gamma}$, where $\tilde{\gamma}$ is the lifting of $\gamma$ with respect to $f_t$ and initial point $g_t^{-1}(w)$. Finally $\Psi$ is a local biholomorphism with the lifting property, then it is a covering mapping.
\endproof  

We want to study the property of the Loewner Range.

\begin{proposition}\label{neccond}
Let $(f_t)_{t\geq0}$ be a Loewner chain of covering mappings and $R(f_t)$ its Loewner range, then 
	
$\bullet$ $R(f_t)$ is pseudoconvex and covered by $\Cn$;
	
$\bullet$ the couple $(f_t(B),R(f_t))$ is Runge for each $t\geq0$;
	
$\bullet$ the morphism of groups $i_t^*:\pi_1(f_t(B))\longrightarrow \pi_1(R(f_t))$ is injective for each $t\geq0$.\\
\end{proposition}
\proof
The first and latter conditions descend respectively from Proposition \ref{fact} and Proposition \ref{p1iso}, while the second one follows from Docquier-Grauert Theorem \cite{DG}.
\endproof

\begin{theorem}\label{genapcov}
For $n\geq2$, let $f:B\twoheadrightarrow\Omega\subseteq\Cn$ be a covering mapping. If there exists a pseudoconvex neighborhood $\Lambda$ of $\Omega$ such that 
	
$\bullet$ $\Lambda$ is holomorphically covered by $\Cn$;
	
$\bullet$ the couple $(\Omega,\Lambda)$ is Runge;
	
$\bullet$ the morphism $i^*:\pi_1(\Omega)\longrightarrow\pi_1(\Lambda)$ is an injective morphism of groups;\\
then $f$ can be approximate uniformly on compacta of $B$ by covering mappings of $\Cn$.
\end{theorem}
\proof
Without loss of generality, we can assume that $f(0)=0$. Let $\Psi:\Cn\twoheadrightarrow\Lambda$ be a covering mapping of $\Lambda$ with $\Psi(0)=0$, and $g$ be the lifting of the inclusion of $\Omega$ into $\Lambda$ with respect to $\Psi$ such that $g(0)=0$, then by Lemma \ref{homeolift} $g$ is univalent.
\[
\begin{tikzcd}[column sep=3pc,row sep=3pc]
\textrm{$B$} \arrow[r,rightarrow,"g"] \arrow[d,rightarrow,"f"] & \textrm{$\Cn$} \arrow[d,rightarrow,"\Psi"] \\
\textrm{$\Omega$} \arrow[r,rightarrow,"i"] & \textrm{$\Lambda$}
\end{tikzcd}
\]
We want to prove that $g(B)$ is Runge. Let $K$ be a compact set of $g(B)$ and consider its polynomial convex hull $\hat{K}:=\hat{K}_{\O(\Cn)}$. Using the Runge-ness of $(\Omega,\Lambda)$ and the Oka-Weil theorem we have
$$\Psi(\hat{K})\subseteq\widehat{\Psi(K)}_{\O(\Lambda)}\subseteq\Omega$$
then
$$\hat{K}\subseteq\Psi^{-1}(\Omega)=\bigcup_{F\in DT(\Psi)} F(g(B)).$$
\textbf{Claim:} For each $F\in DT(\Psi)$, if $F(g(B))\cap g(B)\neq\emptyset$ then $F(g(B))=g(B)$.
\begin{proof}[Proof of Claim]
It is sufficient to show that if $F(g(B))\cap g(B)\neq\emptyset$ then  $F(g(B))\subseteq g(B)$ (replacing $F$ with $F^{-1}$ we obtain the inverse inclusion).
By contradiction assume that there exist $x,y\in g(B)$ such that $F(x)\in g(B)$ but $F(y)\notin g(B)$. Let $\gamma$ be a path in $g(B)$ from $x$ to $y$ and denote with $\tilde{\gamma}$ the lifting of $\Psi\circ\gamma$ with respect to $f$ such that $\tilde{\gamma}(0)=g^{-1}(F(x))$. Now consider the paths $g\circ\tilde{\gamma}$ and $F\circ\gamma$: we have that
$$\Psi\circ(g\circ\tilde{\gamma})=f\circ \tilde{\gamma}=\Psi\circ \gamma=\Psi \circ (F\circ \gamma) $$
and they have the same initial point, then by the uniqueness of the lifting they are the same path. Considering the final point, we obtain
$$F(y)=F(\gamma(1))=g(\tilde{\gamma}(1)) $$
then $F(y)\in g(B)$, and this is a contradiction.		
\end{proof}
Let consider the subgroup $H:=\{F\in DT(\Psi): F(g(B))=g(B)\}$ and $G:=DT(\Psi)/H$, then we have
$$\hat{K}\subseteq\Psi^{-1}(\Omega)=\bigsqcup_{F\in G} F(g(B)),$$
where the union is disjoint.

Now, recalling that every connected component of $\hat{K}$ intersects $K$ \cite[Cor. pag. 186]{Scmplx}, we obtain $\hat{K}\subseteq g(B)$, that implies that $g(B)$ is Runge.

Finally $g$ is an univalent mapping with $g(B)$ Runge, then by Andérsen-Lempert theorem it can be approximate by a sequence $\{\Phi_k\}_{k\in\N}$ of automorphism of $\Cn$, therefore the sequence $\{\Psi\circ\Phi_k\}_{k\in\N}$ is a sequence of covering mappings of $\Cn$ that approximate $f$ uniformly on compact sets of $B$.
\endproof

The proof of Theorem \ref{apprxcov} is now straightforward.

\begin{proof}[Proof of Theorem~\ref{apprxcov}]
It is sufficient to use Theorem \ref{genapcov} taking as $\Lambda$ the Loewner Range. The required assumptions descend from Proposition \ref{neccond}.
\end{proof}

We give an explicit application of the previous theorem.

\begin{example}\label{exball}
In \cite{dF} the covering mappings of $\Bn$ generated by one hyperbolic automorphism are studied in detail: for instance we have the "generalized annulus in higher dimension"
$$f(z_1,\dots,z_n)=\Biggl(\Bigl(\frac{1-iz_1}{1+iz_1}\Bigr)^{\frac{i}{2}}-1, \frac{z_2}{\sqrt{1+z_1^2}},\dots,\frac{z_n}{\sqrt{1+z_1^2}}\Biggr)$$
that is a covering mappings of $\Bn$ with image $$\Omega=\Bigl\{(z_1,\dots,z_n)\in\Cn: r^{-1}<|z_1+1|<r, \sum_{j=2}^{n}|z_j|^2<\sin(-2\log|z_1|)\Bigr\}$$
with $r:=e^{\frac{\pi}{4}}$. We want to prove that $f$ can be approximate by entire covering mappings, showing that $f$ embeds into a Loewner chain of covering mappings. Indeed we define for each $t\geq0$
$$f_t(z_1,\dots,z_n)=\Biggl(\Bigl(\frac{1-iz_1}{1+iz_1}\Bigr)^{\frac{i}{2}e^t}-1, e^t\frac{z_2}{\sqrt{1+z_1^2}},\dots,e^t\frac{z_n}{\sqrt{1+z_1^2}}\Biggr)$$
that is a covering mappings of $\Bn$ with image 
$$\Omega_t=\Bigl\{(z_1,\dots,z_n)\in\Cn: r_t^{-1}<|z_1+1|<r_t, \sum_{j=2}^{n}|z_j|^2<e^{2t}\sin(-2\log|z_1|)\Bigr\}$$
where $r_t=e^{\frac{\pi}{4}e^t}$.
The family $(f_t)_{t\geq0}$ is a Loewner chain of covering mappings with range $\C\backslash\{-1\}\times\C^{n-1}$, and obviously $f$ embeds into $(f_t)_{t\geq0}$, then by Theorem \ref{apprxcov} $f$ can be approximate uniformly on compacta of $\Bn$ by entire covering mappings (similar argument can be used for each covering mappings studied in \cite{dF}).
\end{example}

\section{Other properties of Loewner chains of covering mappings}

In this section we collect other interesting properties of Loewner chains of covering mappings.

It is useful to recall that if $(\phi_{s,t})_{0\leq s\leq t}$ is evolution family, then for Caratheodory Kernel Convergence (see Theorem 3.5 in \cite{ABHK}) we have that for each $0\leq s\leq t$
$$\phi_{s,t}(B)=\bigcup_{u<s} \phi_{u,t}(B) \ \ \text{and}\ \ \phi_{s,t}(B)=int\Bigl(\bigcap_{u>s}\phi_{u,t}(B)\Bigr)_0$$
where with $int(\cdot)_0$ we denote the connected component of the interior part that contains the origin.

\begin{proposition}[Caratheodory Kernel Conv. for Loewner chains of covering maps]\label{CKcov}\mbox{}
\\
Let $(f_t)_{t\geq0}$ be a Loewner chain of covering mappings and denote $\Omega_t:=f_t(B)$. Then for each $t>0$ we have
\begin{equation}\label{CKcond}
\Omega_t=\bigcup_{s<t}\Omega_s\ \ \ \text{and}\ \ \ \Omega_t=int\Bigl(\bigcap_{u>t}\Omega_u\Bigr)_0.
\end{equation}
Conversely, let $\{\Omega_t\}_{t\geq0}$ be a increasing family of domains in $\Cn$ covered by $B$ such that relation (\ref{CKcond}) holds and $\bigcup_{t\geq0}\Omega_t$ is not covered by $B$. Suppose that there exists $f_t$ a covering mappings from $B$ to $\Omega_t$ such that $f_t(0)=0$ and $(df_t)_0=\alpha_t\id$ with $\alpha_t>0$ for each $t\geq0$. Then there exists a strictly increasing continuous function $\beta:\R_{\geq0}\longrightarrow\R_{\geq0}$ with $\beta(0)=0$ and $\lim_{t\rightarrow\infty}\beta(t)=\infty$ such that $(\alpha_0^{-1}f_{\beta(t)})_{t\geq0}$ is a Loewner chain of covering mappings.
\end{proposition}
\proof
We divide the proof in two part

1) We have to prove (\ref{CKcond}): the first one easily descends from the Carathéodory Kernel convergence of the family $(\phi_{s,t})_{0\leq s\leq t}$
$$\Omega_t=f_t(B)=f_t\Bigl(\bigcup_{s<t}\phi_{s,t}(B)\Bigr)=\bigcup_{s<t}f_t(\phi_{s,t}(B))=\bigcup_{s<t}f_s(B)=\bigcup_{s<t}\Omega_s.$$
For the second one, let denote $\Omega:=int(\bigcap_{u>t}\Omega_u)_0$, obviously we have that $\Omega_t\subseteq \Omega$. Conversely, fix $T>t$ and let $V$ be the connected component of  $f^{-1}_T(\Omega)$ that contains 0 (notice that $f_T(V)=\Omega$). Now we want to prove that $V\subseteq\phi_{t,T}(B)$: fix $w\in V$ and let $\tilde{\gamma}$ be a path in $V$ that connects $0$ to $w$, and take $\gamma=f_T\circ\tilde{\gamma}$ that is a path in $\Omega$. Notice that $\tilde{\gamma}$ is the unique lifting of $\gamma$ with respect to $f_T$ with initial point $0$. Now for each $u\in(t,T]$ we have that $\gamma([0,1])\subseteq\Omega_u$, then by Lemma \ref{teclem}
$$w=\tilde{\gamma}(1)\in\tilde{\gamma}([0,1])\subseteq\phi_{u,T}(B)$$
then $V\subset\phi_{u,T}(B)$. Finally, using Carathedory Kernel Convergence for the evolution family
$$V\subset int\Bigl(\bigcap_{u>t}\phi_{u,T}(B)\Bigr)_0=\phi_{t,T}(B)$$
consequently
$$\Omega=f_T(V)=f_t(\phi^{-1}_{t,T}(V))\subseteq\Omega_t.$$

2) First of all, we prove the continuity from above and below of $f_t$. Let $\{t_k\}_{k\in\N}$ be a decreasing sequence that converge to $t$ and consider $T$ such that $T>t_k$ for each $k\in\N$, then for each $0<r<1$
$$\max_{z\in rB}|f_{t_k}(z)|=\max_{z\in rB}|f_T(\phi_{t_k,T}(z))|\leq\max_{z\in rB}|f_T(z)| $$
i.e. the sequence $\{f_{t_k}\}_{k\in\N}$ is uniformly bounded and then it is normal. Therefore there exists a subsequence $\{t_{k_{j}}\}_{j\in\N}$ such that $f_{t_{k_{j}}}$ converges to a function $g$. Using an argument similar to 1), $g$ has range $int(\bigcap_{j\in\N}\Omega_{t_{k_j}})_0=\Omega_t$. We have to prove that $g=f_t$: let $\Phi$ be the lifting of $g$ with respect to $f_t$ such that $\Phi(0)=0$.
\[
\begin{tikzcd}[column sep=3pc,row sep=3pc]
 & \textrm{$B$} \arrow[d,rightarrow,"f_t"] \\
\textrm{$B$} \arrow[ru,rightarrow,"\Phi"] \arrow[r,rightarrow,"g"] & \textrm{$\Omega_t$}
\end{tikzcd}
\]
Let denote $\alpha:=\lim_{j\rightarrow\infty}\alpha_{t_{k_j}}=\inf_{j}\alpha_{t_{k_j}}$. 
Using the chain rule in the expression $f_t\circ\Phi=g$ and noticing that $\alpha_t\leq\alpha$ we obtain
$$(df_t)_0 \cdot (d\Phi)_0=(dg)_0 \ \ \Rightarrow \ \ (d\Phi)_0=\frac{\alpha}{\alpha_t}\id\geq\id$$
that implies for Cartan uniqueness theorem that $\Phi=\id_B$ and then $g=f_t$. Finally, every subsequence of $f_{t_k}$ has the same limit $f_t$, then we have the continuity from above.

The proof of continuity from below is analogous to the argument that we present in the next paragraph, then we omit it.

For the construction of $\beta$, we first observe that the function $t\mapsto\alpha_t$ is strictly increasing, then we can define $\gamma(t):=\log[\alpha_t/\alpha_0]$: the desired function is $\beta:=\gamma^{-1}$. The only thing that we have to prove is that $\lim_{t\rightarrow\infty}\beta(t)=\infty$: by contradiction if $\lim_{t\rightarrow\infty}\beta(t)=:T<\infty$, we define for each $0\leq s\leq t<T$ the function $\phi_{s,t}$ as the unique lifting of the inclusion of $\Omega_{\beta(s)}$ into $\Omega_{\beta(t)}$ that fixes the origin. The family $(\phi_{s,t})_{0\leq s\leq t <T}$ satisfies properties 1)-3) of Proposition \ref{efcov}, hence every map $\phi_{s,t}$ is univalent. Now by normality, for each $0<s<T$ the sequence $\phi_{s,t}$ converges to a function $\phi_{s,T}$ for $t\rightarrow T$ (notice that they are univalent because $d(\phi_{s,T})_0=e^{s-T}\id$), and by the Caratheodory Kernel Convergence we have 
$$\bigcup_{s<T}\phi_{s,T}(B)=B.$$
Finally, we define $f_T$ in the following way: for each $z\in B$ consider a $s\in[0,T)$ such that $z\in\phi_{s,T}(B)$, then $f_T(z):=f_s(\phi^{-1}_{s,T}(z))$. It is easy to check that the definition does not depend from the choice of $s$ and that $f_t$ converges to $f_T$ for $t\rightarrow T$. Now  $f_T$ is a covering mapping from $B$ into $\bigcup_{t\geq0}\Omega_t$ (using a similar argument used for $\Psi$ in the proof of Proposition \ref{fact}), and this is a contradiction. 
\endproof 

Now we want to show some invariant properties of the evolution family associated to a Loewner chains of covering mappings.

\begin{proposition}\label{invphi}
Let $(f_t)_{t\geq0}$ be a Loewner chain of covering mappings and $(\phi_{s,t})_{0\leq s \leq t}$ the associated evolution family. Fix $0\leq s \leq t $ and  $F\in DT(f_t)$, then 
	
$\bullet$ $F(\phi_{s,t}(B))\cap \phi_{s,t}(B)=\emptyset$, or
	
$\bullet$ $F(\phi_{s,t}(B))=\phi_{s,t}(B)$, i.e. the domain $\phi_{s,t}(B)$ is $F$-invariant.
\end{proposition}
\proof
The proof is similar to argument used in  the Claim in Theorem \ref{genapcov}, then we omit it.
\endproof

Example \ref{annuex}, Example \ref{exball} and (univalent) Loewner chains are Loewner chains of covering mappings where the fundamental group of the images does not change. We want to study these kind of chains in detail.

\begin{proposition}\label{invphic}
Let $(f_t)_{t\geq0}$ be a Loewner chain of covering mappings and let be $0\leq s \leq t$, then the following are equivalent

(1) the morphism of groups $i^*_{s,t}:\pi_1(\Omega_s)\longrightarrow \pi_1(\Omega_t)$ is an isomorphism;

(2) for each $F\in DT(f_t)$, $F(\phi_{s,t}(B))=\phi_{s,t}(B)$.
\end{proposition}
\proof
(1)$\Rightarrow$(2) Thanks to Proposition \ref{invphi}, it is sufficient to prove that $F(0)\in\phi_{s,t}(B)$. Let be $\tilde{\gamma}$ a path between $0$ and $F(0)$, and denote $\gamma=f_t\circ\tilde{\gamma}$, (notice that $\gamma(1)=f_t(\tilde{\gamma}(1))=f_t(F(0))=f_t(0)=0$, i.e. it is a closed path). By surjectivity of $i^*_{s,t}$, there exists an homotopy $G:[0,1]\times[0,1]\longrightarrow\Omega_t$ such that $G(0,\cdot)=\gamma$ and $G(1,\cdot)=:\sigma$ is a path in $\Omega_s$. Now $G$ can be lift to an homotopy $\tilde{G}$ in $B$ such that $\tilde{G}(0,\cdot)=\tilde{\gamma}$. Thanks to Lemma \ref{teclem} $\tilde{G}(1,\cdot)$ is a path in $\phi_{s,t}(B)$. Finally, 
$$F(0)=\gamma(1)=\tilde{G}(0,1)=\tilde{G}(1,1)\in\phi_{s,t}(B) $$

(2)$\Rightarrow$(1) It is sufficient to show that every closed path contained in $\Omega_t$ is homotopic to a path in $\Omega_s$: let $\gamma:[0,1]\rightarrow \Omega_t$ be a continuous path such that $\gamma(0)=\gamma(1)=0$, then $\gamma$ can be lifted to $\tilde{\gamma}$ in $B$ such that $f_t\circ \tilde{\gamma}=\gamma$ and $\tilde{\gamma}(0)=0$. By assumption $\tilde{\gamma}(1)\in\phi_{s,t}(B)$, hence there exists an homotopy $\tilde{G}:[0,1]\times[0,1]\rightarrow B$ such that $\tilde{G}(\cdot,0)=\tilde{\gamma}$ and $\tilde{G}(\cdot,1)$ is a path in $\phi_{s,t}(B)$ that connects $0$ to $\tilde{\gamma}(1)$. Now consider the homotopy $G:=f_t\circ\tilde{G}$, we have that $G(\cdot,0)=\gamma$ whereas $G(\cdot,1)$ is a closed path in $\Omega_s$, as needed.
\endproof

\begin{definition}
Let $(f_t)_{t\geq0}$ be a Loewner chain of covering mapping and let $I\subseteq [0,\infty)$ be an interval. We say that $(f_t)_{t\geq0}$ is $I$-\textit{stable} if for each $s,t\in I$ with $s<t$, one (and, consequently, both) of the properties in Proposition \ref{invphic} holds.
Finally, we say simply that $(f_t)_{t\geq0}$ is \textit{stable} if it is $[0,\infty)$-stable.
\end{definition}

\begin{remark}
Let $I\subseteq [0,\infty)$ be an interval and let $(f_t)_{t\geq0}$ be a $I$-stable Loewner chain of covering mappings and $(\phi_{s,t})_{0\leq s\leq t}$ the associated evolution family. Fix $s,t\in I$ with $s<t$ and consider $F\in DT(f_t$), then by Proposition \ref{invphic} the function $\phi_{s,t}^{-1}\circ F\circ\phi_{s,t}$ is an automorphism of $B$ and furthermore 
$$f_s\circ(\phi_{s,t}^{-1}\circ F\circ\phi_{s,t})=f_t\circ F\circ\phi_{s,t}=f_t\circ\phi_{s,t}=f_s.$$
i.e. $(\phi_{s,t}^{-1}\circ F\circ\phi_{s,t})$ is in $DT(f_s)$.
Therefore, for each $F\in DT(f_t$) there exists $G\in DT(f_s$) such that
$$F\circ \phi_{s,t} =\phi_{s,t}\circ G. $$
\end{remark}

Our next aim is to study the chain $(g_t)_{t\geq0}$: 

\begin{definition}
Let $\Gamma$ be a subgroup of Aut($\Cn$).
	
$\bullet$ An univalent function $f:B\longrightarrow\Cn$ is \textit{$\Gamma$-invariant} if $F(f(B))=f(B)$ for each $F\in\Gamma$;
	
$\bullet$ A normal Loewner chain $(f_t)_{t\geq0}$ is a \textit{$\Gamma$-invariant Loewner chain} if $f_t$ is $\Gamma$-invariant for each $t\geq0$.
\end{definition}

In the following remark, we show the link between stable Loewner chains of covering mappings and $\Gamma$-invariant Loewner chains.

\begin{remark}\label{dual}
Let $(f_t)_{t\geq0}$ be a stable Loewner chain of covering mappings and consider $\Psi$ and $(g_t)_{t\geq0}$ as in (\ref{facteq}), then the normal Loewner chain $(g_t)_{t\geq0}$ is $DT(\Psi)$-invariant.
	
Conversely, let $\Psi:\Cn\longrightarrow\Lambda\subseteq\Cn$ be a covering mapping and $(g_t)_{t\geq0}$ be a $DT(\Psi)$-invariant Loewner chain, then $f_t:=\Psi\circ g_t$ is a stable Loewner chain of covering mappings.
\end{remark}

We conclude this section with a simple example in dimension one of a non stable Loewner chain of covering mappings.

\begin{example}
For each $t\geq0$, let
$$\Omega_t:=\C\backslash\bigl(\{z\in\C: \im[z]=0, \re[z]\leq-t-1\}\cup\{-1\}\bigr).$$
The family $\{\Omega_t\}_{t\geq0}$ respects the assumptions of Proposition \ref{CKcov}, then there exists an unique Loewner chain of covering mappings $(f_t)_{t\geq0}$, a continuous increasing function \\$\beta:[0,\infty)\longrightarrow[0,\infty)$ and $\alpha_0>0$ such that $f_{t}(\D)=\alpha_0\Omega_{\beta^{-1}(t)}$. Now $f_0$ is a biholomorphism, instead $f_t$ is a covering mapping with $\pi_1(f_t(\D))\cong\Z$ for each $t>0$. Then  $(f_t)_{t\geq0}$ is $(0,\infty)$-stable but it is not stable.
\end{example}

\section{Embedding problem in one dimension}

In this section we investigate the problem of embedding of covering mappings in Loewner chains. As in the univalent case, we said that a covering mapping $f$ on $B$ embeds into a \textit{Loewner chain of covering mapping} if there exists a Loewner chain of covering mappings $(f_t)_{t\geq0}$ such that $f_0=f$. 

In the univalent case, we have this powerful Theorem proved by Pommerenke in \cite{PMK}.

\begin{theorem}\label{emb1d}
Let $f:\D\longrightarrow\C$ be a normalized univalent function, then $f$ embeds into a Loewner chain.
\end{theorem}

First of all, we need the following topological result.

\begin{proposition}\label{complannu}
Let $\Omega\subseteq\R^2$ be a domain with $\pi_1(\Omega)\cong\Z$, then $\R^2\backslash\Omega$ has only one bounded connected component.
\end{proposition}
\proof
First of all, we embedded $\R^2$ into its one-point compactification: it is sufficient to prove that $\Omega\subseteq S^2$ divides the sphere in two connected components. Furthermore, it is a well known fact in Riemann surfaces theory that $\Omega$ has to be homeomorphic to the annulus $A:=\{x\in\R^2: 1<||x||<2\}$ \cite{AB}. Denote with $F$ an homeomorphism between $A$ and $\Omega$. Now, for each $r\in (1,2)$ we consider the curve $\gamma_r:=F(\{||x||=r\})\subseteq S^2$. Fix a point $x_0\in S^2\backslash\Omega$, for each $r\in(1,2)$ by the Jordan theorem $\gamma_r$ divides $S^2$ in two connected components: denote by with $A_r$ the component that contains $x_0$, and by $B_r$ the other one. Now we have that
\begin{equation}\label{settop}
S^2\backslash\Omega=S^2\backslash\bigcup_{r\in(1,2)}\gamma_r=\bigcap_{r\in(1,2)}(S^2\backslash\gamma_r)=\bigcap_{r\in(1,2)}(A_r\cup B_r)=\Bigl(\bigcap_{r\in(1,2)}A_r\Bigr)\cup\Bigl(\bigcap_{r\in(1,2)}B_r\Bigr).
\end{equation}
we observe that for each $r\neq r'$ we have $\overline{A_r}\subseteq A_{r'}$ or $\overline{A_{r'}}\subseteq A_{r}$. Then 
$$A:=\bigcap_{r\in(1,2)}A_r=\bigcap_{r\in(1,2)}\overline{A_r}$$ is the intersection of a decreasing family of connected compact sets , hence it is a non empty connected compact set (and the same holds for $B:=\bigcap_{r\in(1,2)}B_r$). Finally, by (\ref{settop})
$$S^2\backslash\Omega=A\cup B$$
and we are done.
\endproof

Finally we can prove the embedding theorem in one dimension.

\begin{theorem}\label{emb1dcov}
Let $f:\D\twoheadrightarrow\Omega\subseteq\C$ be a covering mapping, then it embeds into a Loewner chain of covering mappings if and only if $\pi_1(\Omega)\cong\{0\} $ or $\pi_1(\Omega)\cong\Z$.
\end{theorem}
\proof
\textbf{Necessary condition:} consider $\Lambda$ the Loewner range of the chain. It is a standard fact in theory of Riemann surfaces that the domains of $\C$ covered by the complex plain are $\C$ itself and $\C\backslash\{point\}$ \cite{AB}, and then the group $\pi_1(\Lambda)$ is trivial in first case and isomorphic to $\Z$ in the second one. We conclude using that $\pi_1(\Omega)$ is isomorphic to a subgroup of $\pi_1(\Lambda)$ (Proposition \ref{p1iso}).

\textbf{Sufficient condition:} If $\pi_1(\Omega)$ is trivial then by covering spaces theory $f$ is univalent (and then we are in the classical Loewner theory) and so we can conclude with the aim of Theorem \ref{emb1d}. Otherwise if $\pi_1(\Omega)\cong\Z$ the domain $\Omega$ is an annulus: we want to construct a continuous  increasing family of annuli and use Proposition \ref{CKcov}. Thanks to Proposition \ref{complannu}, the complement of $\Omega$ in $\C$ has only one bounded component $B$. Now if $\C\backslash\Omega$ has unbounded component, then $\tilde{\Omega}:=\Omega\cup B$ is a simply connected domain strictly contained in $\C$, and then we can construct an increasing family of simply connected domains $(\tilde{\Omega}_t)_{t\in[0,T]}$ such that $\tilde{\Omega}_0=\tilde{\Omega}$ and $\tilde{\Omega}_T=\C\backslash B$. We set $\Omega_t:=\tilde{\Omega}_t\backslash B$ for $t\in[0,T)$: we constructed the first part of the chain, eliminating the unbounded parts of $\C\backslash\Omega$. 
Now consider a point $z_0\in B$ and the biholomorphism $F(z)=1/(z-z_0)$: we have that $\tilde{\Omega}_T=F(\C\backslash B)\cup\{0\}$ is a simply connected domain of $\C$, then it can be embedded into a continuous  increasing family of simply connected domains $(\tilde{\Omega}_t)_{t\in[T,\infty)}$. Finally, we define for each $t\geq T$ the annuli $\Omega_t:=F^{-1}(\tilde{\Omega}_t\backslash\{0\})$. Now the family $\{\Omega_t\}_{t\geq0}$ respects the assumption of Proposition \ref{CKcov} and $\Omega_0=\Omega$, then $f$ embeds into a Loewner chain of covering mappings.
\endproof

Finally using Remark \ref{dual} we obtain the following 

\begin{corollary}
Let $f:\D\longrightarrow \C$ be a normalized univalent function invariant with respect to $F(z):=z+i$, i.e. $F(f(\D))=f(\D)$, then $f$ can be embedded into a $F$-invariant Loewner chain.
\end{corollary}

We conclude this section with an application in several complex variables. 

\begin{theorem}\label{aprxpol}
Let $n\geq2$ and $F:\D^n\longrightarrow \Cn$ be a function of the form
$$F(z_1,\dots,z_n)=(f^{(1)}(z_1),\dots,f^{(n)}(z_n)) $$
where $f^{(j)}:\D\twoheadrightarrow\Omega_j\subseteq\C$ is univalent or a covering mapping with $\pi_1(\Omega_j)\cong\Z$ for each $j=1,\dots,n$. Then $F$ can be approximated uniformly on compacta of $\D^n$ by entire covering mappings.
\end{theorem}
\proof
First of all, it is easy to see that $F$ is a covering mapping from $\D^n$ to $\Omega_1\times\dots\times\Omega_n$. By Theorem \ref{emb1dcov}, for each $j=1,\dots,n$ the mapping $f^{(j)}$ can be embedded into a Loewner chain of covering mappings $(f^{(j)}_t)_{t\geq0}$, then we define the chain 
$$F_t(z_1,\dots,z_n):=(f^{(1)}_t(z_1),\dots,f^{(n)}_t(z_n)). $$
Now $(F_t)_{t\geq0}$ is a Loewner chain of covering mappings on $\D^n$ with associated evolution family 
$$\Phi_{s,t}(z_1,\dots,z_n):=(\phi^{(1)}_{s,t}(z_1),\dots,\phi^{(n)}_{s,t}(z_n))$$
where $(\phi^{(j)}_{s,t})_{0\leq s\leq t}$ are the evolution families associated to the chains $(f^{(j)}_t)_{t\geq0}$, for each $j=1,\dots,n$. Obviously $F$ embeds into $(F_t)_{t\geq0}$ and then we can conclude using Theorem \ref{apprxcov}.
\endproof

\section{Open questions}

We conclude with some open questions.

\begin{question}
Let $\Gamma\leq$Aut($\Cn$) be a subgroup and $f:B\longrightarrow\Cn$ be a $\Gamma$-invariant univalent function. Can $f$ be embedded into a Loewner chain? Can it be embedded into a $\Gamma$-invariant Loewner chain?
\end{question}

We notice that the assumption of pseudoconvexity is used only in  Proposition \ref{neccond} and Theorem \ref{genapcov}. 

\begin{question}
Can we eliminate the pseudoconvexity assumption? Is the image of a covering mapping that embeds into a Loewner chain of covering mappings necessary a pseudoconvex domain?
\end{question}

\end{document}